\documentclass{amsart}

\usepackage{amsfonts}
\usepackage{cases}
\usepackage{mathrsfs}
\usepackage{bbm}
\usepackage{amssymb}
\usepackage{txfonts}
\usepackage{amscd}
\usepackage{amsfonts,latexsym,amsmath,amsthm,amsxtra,mathdots}
\usepackage[bookmarks,colorlinks]{hyperref}
\usepackage[all,cmtip]{xy}
\RequirePackage{amsmath} \RequirePackage{amssymb}
\usepackage{color}
\usepackage{colordvi}
\usepackage{multicol}
\usepackage{tikz}
\usepackage{hyperref}
\usepackage{graphicx}
\usepackage{amsmath}
\usepackage{amsmath,amscd}
\usetikzlibrary{matrix,arrows}

\DeclareMathOperator{\Lk}{Lk}

\newcommand{\id}{{{\rm id}}}

     \newcommand{\BN}{{\mathbb {N}}}
     
    \newcommand{\BQ}{{\mathbb {Q}}}

     \newcommand{\BZ}{{\mathbb {Z}}}

     \newcommand{\Aut}{{\mathrm{Aut}}}

    \newcommand{\Hom}{{\mathrm{Hom}}} 
    \newcommand{\Ind}{{\mathrm{Ind}}}

    \newcommand{\im}{{\mathrm{im}}}

\def\-{^{-1}}

\newcommand{\delete}[1]{}

    \theoremstyle{plain}

\newtheorem{thm}{Theorem}[section]
\newtheorem{defn}[thm]{Definition} 
\newtheorem{lem}[thm]{Lemma}
\newtheorem{prop}[thm]{Proposition}
\newtheorem{cor}[thm]{Corollary}
\newtheorem{rem}[thm]{Remark}

\newtheorem*{thmA}{Theorem A}
\newtheorem*{thmB}{Theorem B}
\newtheorem*{rem*}{Remark}

    \numberwithin{equation}{section}

\DeclareMathOperator{\Fix}{Fix}
\DeclareMathOperator{\wCM}{wCM}

\DeclareMathOperator{\trunc}{trunc}
\providecommand{\FI}{\ensuremath\mathsf{FI}}
\providecommand{\clst}[1]{\ensuremath\overline{\mathrm{St}}(#1)}
\providecommand{\GH}{\ensuremath\mathcal G_{\mathcal H}}
\providecommand{\CH}{\ensuremath\mathcal C_{\mathcal H}}

\newenvironment{inlinecond}[1]{%
  \begin{enumerate}%
  \renewcommand*{\theenumi}{#1}%
  \renewcommand*{\labelenumi}{#1}%
  \item%
}{%
  \end{enumerate}%
}

\newenvironment{cond}[1]{%
  \vspace{1.2ex}\begin{enumerate}%
  \renewcommand*{\theenumi}{#1}%
  \renewcommand*{\labelenumi}{#1}%
  \item%
}{%
  \end{enumerate}\vspace{1.2ex}%
}

\begin{document}

\title{Stability results for Houghton groups}

\author{Peter Patzt}
\address{Institut f\"ur Mathematik, Freie Universit\"at Berlin, Germany}
\email{peter.patzt@fu-berlin.de}

\author{Xiaolei Wu}
\address{Institut f\"ur Mathematik, Freie Universit\"at Berlin, Germany}
\curraddr{Max Planck Institute for Mathematics, Bonn, Germany}
\email{hsiaolei.wu@mpim-bonn.mpg.de}


\date{February, 2016}

\keywords{homology stability; representation stability; Houghton groups.}

\begin{abstract}
We
prove homological stability for a twisted version of the Houghton groups and their multidimensional analogues. Based on this, we can describe the homology of the Houghton groups and that of their multidimensional analogues over constant noetherian coefficients as an essentially finitely generated $\FI$-module. 
\end{abstract}

\maketitle

\section*{Introduction}
The Houghton groups were first introduced in \cite{H} by Houghton. In
\cite{B}, K. Brown proved that Houghton's group $\mathcal{H}_n$ is of
type $FP_{n-1}$ but not $FP_n$.  The group $\mathcal{H}_n$ can be
defined as follows (cf. \cite[Section 5]{B}).

Let $\BN$ be the set of positive integers, and $[n] = \{1,2, \dots, n\}$. Let $\mathcal{H}_n$ be the group of all permutations (self bijections) $g$ of $\BN \times [n]$ such that on each copy of $\BN$, $g$ is eventually a translation. More precisely, we require the following condition.

\begin{cond}{($\ast$)}\label{cond:Houghton}
\textit{There is an $n$-tuple $(d_1,d_2, \dots,d_n ) \in \BZ^n$ such
that for each $i \in [n]$ one has $g(x,i)=(x+d_i,i)$ for sufficiently
large $x\in \BN$.}
\end{cond}

We also define a twisted version $\tilde{\mathcal{H}}_n$  of
$\mathcal{H}_n$ as follows. An element $g \in \tilde{\mathcal{H}}_n$ is
a permutation of $\BN\times [n]$ such that the following condition it true.

\begin{cond}{($\tilde\ast$)}\label{cond:twisted Houghton}
\textit{There is an $n$-tuple $(d_1,d_2, \dots,d_n ) \in \BZ^n$ and
$\sigma \in \mathfrak S_n$ such that for each $i \in [n]$ one has
$g(x,i)=(x+d_{i},\sigma(i))$ for sufficiently large $x\in \BN$.}
\end{cond}

We can embed the symmetric group $\mathfrak S_n$ into the set of permutations
of $\BN\times [n]$ by only acting on $[n]$. That is to say, an element
of $\sigma \in \mathfrak S_n$ acts on $(x,i) \in  \BN \times [n] $ by $\sigma
(x,i) = (x,\sigma(i))$. Then $\tilde{\mathcal{H}}_n$ is generated by
$\mathcal{H}_n$ and $\mathfrak S_n$. In fact, $\mathcal{H}_n$ is a normal
subgroup of $\tilde{\mathcal{H}}_n$ and  $\tilde{\mathcal{H}}_n  \cong
{\mathcal{H}}_n \rtimes \mathfrak S_n$. Since $\mathcal{H}_n$  is a finite
index subgroup of $\tilde{\mathcal{H}}_n$, the twisted Houghton group
$\tilde{\mathcal{H}}_n$ has the same finiteness properties as
$\mathcal{H}_n$ (cf. \cite[Chapter VIII, Proposition 5.1]{B1}).

The inclusion map $\BN \times [n] \subset \BN \times [n+1]$ induces a
map from $\tilde{\mathcal{H}}_n$ to $\tilde{\mathcal{H}}_{n+1}$. We will
prove that these groups satisfy homological stability. Indeed we want to
prove this for a multidimensional version of the Houghton groups. These
groups were defined by Bieri and Sach recently in \cite{BH}, where they
proved the multidimensional version has similar finiteness properties as
the
original Houghton groups.

Let us first define $\tilde{\mathcal{H}}_{k,n}$, the $k$-dimensional
version of the twisted Houghton group. This shall be a subgroup of
permutations on $\BN^k \times [n]$. We will call a subset $X\subset \BN^k$
an
\emph{$r$-dimensional ray} if there is a point $x\in \BN^k$ and an
$r$-subset $T\subset [k]$ such that
\[ X = \{ y \in \BN^k \mid \begin{array}{l} \forall j \in T\colon y_j
\ge x_j \\ \forall j \not \in T \colon  y_j = x_j \end{array}\}.\]
Let $\tilde{\mathcal H}_{k,n}$ be the group of all permutations that are
translations on all rays of a finite partition of $\BN^k \times [n]$ into
rays. By a translation we mean a map $f\colon X\to \BN^k \times [n]$ given
by $f(x,i) = (x+d, i + d_0)$ for $d\in \BZ^k$ and $d_0\in \BZ$.  Note that $0$-dimensional rays are just points. 
In Bieri and Sach's notion, this is the group $Pet(\BN^k \times [n])$  where $\BN^k \times [n] \subset \BZ^{k+1}$, and the structure group  $\BZ^{k+1}$  acts on the lattice $\BZ^{k+1}$ as translations.


Every finite partition of $\BN^k\times[n]$ must contain exactly one
$k$-dimensional ray for every copy of $\BN^k$. Therefore there is a
surjection $\tilde{\mathcal{H}}_{k,n} \to \mathfrak S_n$. Define
${\mathcal{H}}_{k,n}$ as the kernel of that map. Note that this short
exact sequence splits. Again from the inclusion of $\BN^k \times [n]
\subset \BN^k \times [n+1]$ we get inclusion maps
$\tilde{\mathcal{H}}_{k,n} \to \tilde{\mathcal{H}}_{k,n+1}$ and
${\mathcal{H}}_{k,n} \to {\mathcal{H}}_{k,n+1}$. Note that
$\tilde{\mathcal{H}}_{1,n} = \tilde{\mathcal{H}}_n$ and
${\mathcal{H}}_{1,n} = \mathcal H_n$.

We now formulate our homological stability result for the twisted
Houghton groups.

\begin{thmA}
The induced map
\[H_i(\tilde{\mathcal{H}}_{k,n}; \BZ )  \longrightarrow H_i(\tilde{\mathcal{H}}_{k,n+1}; \BZ )\]
is surjective if  $i \leq \frac{n-1}{2}$ and injective if  $i \leq \frac{n-2}{2}$.
\end{thmA}

\begin{rem*} Here we restrict our main result to the constant
coefficient $\BZ$ case. Nevertheless, the theorem also holds for some
general coefficients  by applying Theorem A in \cite{RWW}. \end{rem*}

\begin{rem*}  In \cite{BCMR}, it was showed that $Aut ({\mathcal{H}}_n)$
is isomorphic to $\tilde{\mathcal{H}}_n$. Hence Theorem A can also be
understood as homology stability phenomenon for  $Aut ({\mathcal{H}}_n)$.
\end{rem*}

We have a natural action of the symmetric groups on the homology groups
of the Houghton groups. Therefore we should not expect homological
stability. Using a result of Putman and Sam \cite{PS} we can prove a
representation stability result though. There have been different
notions of representation stability. One that seems to imply most of
them in different contexts is the structure of a finitely generated
$\FI$-module.

\begin{thmB}
Let $R$ a commutative noetherian  ring. Then for every $i,k\in \BN$
there is an $\FI$-module $V$, given by $V_n = H_i(\mathcal H_{k,n}; R)$,
which is essentially finitely generated, i.e. there is a finitely
generated $
FI$-module $W$ and a map $W\to V$ such that $W_n\to V_n$ is surjective
for all large enough $n\in \BN$.
\end{thmB}
\begin{rem*}
Setting $R= \BQ$ this immediately gives uniform representation stability
by a theorem (cf.\ \cite[1.13]{CEF}) of Church, Ellenberg, and Farb.
Only for dimension $k=1$ one can see this directly,
  using the short exact sequence $1 \rightarrow \mathfrak S_{\infty}
\rightarrow \mathcal H_n \rightarrow \BZ^{n-1} \rightarrow 1 $ and
$H_i(\mathfrak S_{\infty} ;\BQ) \cong \{0\}$.
\end{rem*}

The paper is organised as follows. In Section \ref{section:hom stab} we
prove Theorem A and in Section \ref{section:rep stab} we prove Theorem
B. In more detail in the first section we recall definitions and results
from the categorical framework for homological stability in \cite{RWW}.
We prove Theorem A by constructing a homogeneous category for the
twisted Houghton group $\tilde{\mathcal{H}}_n$ and proving its
associated simplicial complex is highly connected by applying a
generalization of a proposition of Hatcher and Wahl (cf.\ \cite[Prop.\
3.5]{HW}). In the second section we quickly give the necessary
background for a result of Putman and Sam (cf.\ \cite[Thm.\ 5.13]{PS}).
We then modify this result slightly to  conclude  Theorem B.

\textbf{Acknowledgements.} The first author was supported by the Berlin Mathematical School. The second author was supported by the Point fellowship from the Dahlem Research School. The authors also want to thank Elmar Vogt and Nathalie Wahl for helpful discussions. The proof Corollary \ref{cor:H3} was considerably shortened after Nathalie Wahl pointed out the similarities to complete join complexes.

\section{Homology stability}\label{section:hom stab}

We begin with a summary of the axiomatized approach to homological stability given by Randal-Williams and Wahl. The definitions and results concerning homogeneous categories are taken from \cite{RWW}.  The reader is encouraged to consult the cited paper for more details.

\begin{defn}[{\cite[1.2]{RWW}}]
Let a monodial category $(\mathcal{C}, \oplus, 0)$ be called \emph{homogeneous} if $0$ is initial in $\mathcal{C}$ and if the following two properties hold.
\begin{inlinecond}{\textbf{H1}} $\Hom(A,B)$ is a transitive $\Aut(B)$-set under composition.\end{inlinecond}
\begin{inlinecond}{\textbf{H2}} The map $\Aut(A) \rightarrow \Aut(A \oplus B)$ taking $f$ to $f \oplus \id_B$ is injective with image 
\[\Fix(B) := \{ \phi\in \Aut(A\oplus B) \mid \phi\circ (\imath_A\oplus \id_B) = \imath_A \oplus \id_B\}\]
where $\imath_A\colon 0 \to A$ is the unique map.\end{inlinecond}
\end{defn}


For a homological stability result on a sequence of automorphism groups of a homogeneous category, the connectivity of a certain simplicial complex that we define next is needed.

\begin{defn}[{\cite[2.8+2.2]{RWW}}] \label{defn:Sn(X,A)}
Let $A,X$ be objects of a homogeneous category $(\mathcal{C}, \oplus,0)$. For $n \geq 1$, let $S_n(A,X)$ denote the simplicial complex whose vertices are the maps $f\colon X \to A\oplus X^{\oplus n}$ and whose $p$-simplices are $(p+1)$-sets $\{ f_0, \dots, f_p\}$ such that there exists a morphism $f\colon X^{\oplus p+1} \to A\oplus X^{\oplus n}$ with $f \circ i_j = f_j$ for some order on the set, where
\[ i_j =  \imath_{X^{\oplus j}} \oplus \id_X \oplus \imath_{X^{\oplus p-j}}  \colon X = 0 \oplus X \oplus 0 \longrightarrow X^{\oplus p+1}. \]

Also define the following property for a fixed pair $(A,X)$ and a \emph{slope} $k\ge2$.
\begin{inlinecond}{\textbf{LH3}}\label{cond:H3}
\textit{For all $n \geq 1$, $S_n(A,X)$ is $(\frac{n-2}{k})$-connected.}\end{inlinecond}
\end{defn}

\begin{defn}[{\cite[2.5]{RWW}}]\label{defn:standard}
Let $A,X$ be objects of a homogeneous category $(\mathcal{C}, \oplus,0)$. $\mathcal C$ is called \emph{locally standard} at $(A,X)$ if it satisfies the following two conditions.
\begin{inlinecond}{\textbf{LS1}}\label{cond:LS1}
\textit{The morphisms $\imath_A \oplus \id_X \oplus \imath_X$ and $\imath_{A\oplus X} \oplus \id_X$ are distinct in $\Hom(X,A\oplus X^{\oplus2})$.}\end{inlinecond}
\begin{inlinecond}{\textbf{LS2}}\label{cond:LS2}
\textit{For all $n\ge 1$, the map $\Hom(X,A \oplus X^{\oplus n-1}) \to \Hom(X, A \oplus X^{\oplus n})$ taking $f$ to $f \oplus \imath_X$ is injective.}\end{inlinecond}
\end{defn}

\begin{rem}\label{rem:semisimplicial}
In \cite[2.2]{RWW} actually a different semisimplicial complex is used to define \ref{cond:H3}, but our \ref{cond:H3} implies theirs if $\mathcal C$ is symmetric and locally standard at $(A,X)$. This is shown in \cite[2.9+2.10]{RWW}.
\end{rem}

We are now ready to quote the theorem we will use.
\begin{thm}[{\cite[3.1]{RWW}}] \label{thm:hom stab}
Let $(\mathcal{C}, \oplus,0)$ be a symmetric homogeneous category satisfying \ref{cond:H3}, \ref{cond:LS1}, and \ref{cond:LS2} for a pair $(A,X)$ with slope $k \geq 2$. Then the map
\[H_i(\Aut(A\oplus X^{\oplus n}); \BZ )  \longrightarrow H_i(\Aut(A\oplus X^{\oplus n+1}); \BZ )\]
induced by the natural inclusion map is surjective if  $i \leq \frac{n}{k}$, and injective if  $i \leq \frac{n-1}{k}$.
\end{thm}

Finally we want to quote a construction theorem that will allow us to build a category in which $\Aut(X^{\oplus n})$ are the twisted Houghton groups.

\begin{thm}[{\cite[1.6+1.8+1.10]{RWW}}]\label{thm:hom cat construction}
Given a symmetric monoidal groupoid $(\mathcal G, \oplus, 0)$ with $\Aut(0) = \{\id\}$ and the map $\Aut(A) \to \Aut(A\oplus B)$ sending $f$ to $f\oplus\id_B$ is injective for all objects $A,B$ in $\mathcal G$.  Assume furthermore that the underlying monoid has no zero divisors and is cancellative. Then there is a  symmetric homogeneous category $\mathcal C$, which is defined on the same elements as $\mathcal G$ with homomorphism sets $\Hom_{\mathcal C}(A, B\oplus A) = \Aut(B\oplus A)/\Aut(B)$ and empty if the codomain is not isomorphic to any such sum. 
\end{thm}

Let us fix a dimension $k\in \BN$. We want to prove homological stability for the $k$-dimensional twisted Houghton groups in this section. In order to apply Theorem \ref{thm:hom stab}, we need to introduce a symmetric homogeneous category that can be constructed using Theorem \ref{thm:hom cat construction}. Then it suffices to prove \ref{cond:H3}, \ref{cond:LS1}, and \ref{cond:LS2}.

Let $\GH$ be the groupoid whose objects are the nonnegative integers such that its morphisms are all automorphisms with $\Aut(0) = \{\id\}$ and $\Aut(n) = \tilde{\mathcal H}_{k,n}$ for $n\ge1$. This groupoid is symmetric monoidal with the usual addition on the integers, which has no zero divisors and is cancellative. This can be seen with the map
\[ \Aut(m) \times \Aut(n) \longrightarrow \Aut(m+n)\]
where we want $g\in\Aut(m)$ to act as usual on $\BN^k\times \{1,\dots, m\}$ and $g'\in \Aut(n)$ to act on $\BN^k\times \{m+1,\dots,m+n\}$. Since $g$ and $g'$ commute, this defines a monoidal structure on $\GH$. The map
\[ \left((x,i) \mapsto \begin{cases} (x,i+n) &\text{for }i\le m\\ (x,i-m)&\text{for }i>m\end{cases} \right)\in \Aut(m+n)\]
defines a symmetry.

Let $\CH$ be the homogeneous category constructed by Theorem \ref{thm:hom cat construction}. The next lemma will help us understand the morphism sets $\Hom_{\CH}(m,n)$ better.

\begin{lem}\label{lem:hom(m,n)}
Let $m < n$, then $\Hom_{\CH}(m,n)$ can be naturally identified with the set of injections $g$ from $\BN^k \times [m]$ to $\BN^k \times [n]$ which are translations on every ray of a finite partition of $\BN^k\times [m]$ into rays.
\end{lem}

\begin{proof}
From $m < n$, we get $\Hom_{\CH}(m,n) = \tilde{\mathcal H}_{k,n}/\tilde{\mathcal H}_{k,n-m}$ by Theorem \ref{thm:hom cat construction}. Let us define a map $F$ from these injections  to $ \tilde{\mathcal H_{k,n}}/\tilde{\mathcal H}_{k,n-m}$.  Given any injection from $\BN^k \times [m]$ to $\BN^k \times [n]$ which are translations on every ray of a finite ray partition of $\BN^k\times [m]$, one easily extends it to a permutation of $\BN^k\times [n]$ which is a translation on every ray of a partition of $\BN^k \times [ n]$ into rays. Given two elements $g_1,g_2\in \tilde{\mathcal H}_{k,n}$ extending the same injection then they coincide on the first $m$ copies of $\BN^k$. Let $h=g_2^{-1}g_1$ which acts trivially on $\BN^k\times [m]$. Thus that $h\in \tilde{\mathcal H}_{k,n-m}$. Thus the given map  $F$ is well defined.

Given an element $g\in  \tilde{\mathcal H}_{k,n}$, we can restrict it to $\BN^k\times [m]$ to get an injection that is a translation on rays. Thus $F$ is surjective. Because every element in a coset has the same restriction the map $F$ is injective.
\end{proof}

Let us choose $X =A= 1$. From the previous lemma $\CH$ is clearly locally standard at $(A,X)$. We can apply Theorem \ref{thm:hom stab} to get Theorem A, if  the corresponding simplicial complex $S_n(1,1) \cong S_{n+1}(0,1)$ (cf. Definition \ref{defn:Sn(X,A)}) is $\frac{n-2}{2}$-connected. Let us abbreviate $S_n(0,1)$ by $S_n$.

Let us prove a few properties about simplicial complexes and come back to $S_n$ again later.

\begin{defn} \label{defn:wCM}
A simplicial complex $K$ is called \emph{weakly Cohen-Macaulay of dimension $n$} if it is $(n-1)$-connected and the link of any $p$-simplex is $(n-p-2)$-connected. In this case, we write $\wCM(K) \geq n$.
\end{defn}

\begin{rem}
Note that $(-1)$-connected is defined to say non-empty. This implies that $K$ is at least of dimension $n$.  Note also that a weakly Cohen-Macaulay complex of dimension $n$ is  weakly Cohen-Macaulay of dimension $m\le n$.
\end{rem}

\begin{defn}
Let $\pi\colon Y \to X$ be a surjective simplicial map between simplicial complexes. Let $S$ be a subset of the vertices of $Y$. We call a section $\rho\colon X \to Y$ of $\pi$ an \emph{$S$-section} if for all simplices $\sigma$ in the span of $S$ and every simplex $\tau$ in $X$ we have
\[ \tau \subset \Lk_X \pi \sigma \Longleftrightarrow \rho \tau \subset \Lk_Y \sigma.\]
We call $\pi$ a \emph{fin-retraction} if there exists an $S$-section for every finite set of vertices of $Y$. 
\end{defn}

\begin{rem}
When Hatcher and Wahl call $Y$ a complete join complex over $X$ (cf.\ \cite[3.2]{HW}), there is an $S$-section where $S$ is the set of all vertices of $Y$. In particular, then $Y\to X$ is a fin-retraction. The proof of our next proposition is a generalization of \cite[3.5]{HW}.
\end{rem}

\begin{prop}\label{prop:f-join}
If $\pi\colon Y\to X$ is a fin-retraction and $\wCM(X) \ge n$, then $Y$ is $(n-1)$-connected. If $\pi$ is simplexwise injective (ie.\ links map to links), also $\wCM(Y) \ge n$.
\end{prop}

\begin{proof}
We first prove that $Y$ is $(n-1)$-connected. We need to prove that every map $f\colon S^k\to Y$ with $k\le n-1$ homotopes to a constant map. We may assume that $f$ is simplicial. Let $S$ be the set of images of the vertices in $S^k$ and $\rho\colon X\to Y$  an $S$-section. Then $\rho X$ is an isomorphic image of $X$ in $Y$. If we can homotope $f$ to land in $\rho X$, we proved $(n-1)$-connectedness. Call a simplex of $Y$ \emph{bad} if it lies in the complement of $\rho X$.  Let $\sigma$ be a simplex of $S^k$ with maximal dimension $q$ such that $f(\sigma)$ is a bad simplex of $Y$, say of dimension $p\le q$. By maximality, $f$ maps the link of $\sigma$ to $\rho X$. This implies that every simplex in $\Lk_Y f\sigma$ is the image of a simplex $\tau$ of $X$ under $\rho$. But $\rho \tau$ will be sent to $\pi \rho \tau = \tau$ in $X$, which is in $\Lk_X \pi f \sigma$ because $\rho$ is an $S$-section. In fact $\Lk_Y f\sigma$ is isomorphic to $\Lk_X \pi f \sigma$ because $\rho$ is an $S$-section.
This link is $(n-p-2)$-connected by assumption on $X$ and $\Lk_{S^k} \sigma \cong S^{k-q-1}$. As $k-q-1 \le n- p -2$, there exists a map $F\colon D^{k-q} \to \Lk_Y f\sigma$ extending $f|_{\Lk \sigma}$. By the (relative) simplicial approximation theorem, we can extend the simplicial structure of $\Lk_{S^k} \sigma$ to $D^{k-q}$ and assume that $F$ is simplicial. Now $H:= F * f|_{\sigma}\colon D^{k-q} * \sigma \to \clst{f\sigma}$ exists. The boundary of the ball $D^{k-q} * \sigma$ is $\partial D^{k-q} * \sigma \cup D^{k-q} *\partial \sigma$. Therefore $H$ defines a homotopy from $f|_{\clst \sigma} \colon \clst \sigma = \partial D^{k-q} * \sigma \to \clst{ f\sigma}$ to $F*f|_{\partial \sigma}\colon D^{k-q} * \partial \sigma \to \clst{f\sigma}$. This defines a new map $f'$ homotopic to $f$ with fewer maximal simplices whose images are bad. Note that the simiplical structure on $S^k$  outside $ \clst \sigma $ has not changed, however, the simplicial structure on $ \clst \sigma $ has changed from $\Lk_{S^k} \sigma  ~*  ~\sigma$ to $D^{k-q} ~ *  ~\partial \sigma$. After finitely many iterations no bad simplices remain which shows that $Y$ is $(n-1)$-connected.

Assume $\pi$ is simplexwise injective. Let $\sigma$ be a $p$-simplex in $Y$, then we need to prove that $\Lk_Y \sigma$ is $(n-p-2)$-connected.  Note that the link of a $p$-simplex in $X$ is $(n-p-2)$-connected. Now we consider the restriction of $\pi$ to $\Lk_Y \sigma$ which maps to $\Lk_X \pi \sigma$ because $\pi$ is simplexwise injective. Let $\rho\colon X \to Y$ be an $S$-section where $S$ is the set of vertices of $\sigma$. That means for all simplices $\tau$ in $X$
\[ \tau \in \Lk_X \pi \sigma  \iff \rho \tau \in \Lk_Y \sigma.\]
In conclusion $\pi \colon \Lk_Y \sigma  \to \Lk_X \pi \sigma$ is surjective. Now we want to show that $\pi\colon \Lk_Y \sigma\to \Lk_X \pi \sigma$ is a fin-retraction. Let $S$ be a finite set of vertices in $\Lk_Y \sigma$, let $\sigma'$ be a simplex whose vertices are in $S$, and let $\tau$ be a simplex in $\Lk_X \pi \sigma$. Then the following equivalences hold.
\[ \tau \subset \Lk_{\Lk_X \pi \sigma} \pi\sigma' \iff \tau \subset \Lk_X \pi( \sigma' * \sigma)  \iff \rho \tau \subset \Lk_Y \sigma' * \sigma \iff \rho \tau \subset \Lk_{\Lk_Y \sigma} \sigma'\]
This proves that $ \wCM(Y) \ge n$ if $\wCM(X) \ge n$. 
\end{proof}

Returning to $S_n$, let us define a map from $\pi\colon S_n\to \Delta^{n-1}$. By Lemma \ref{lem:hom(m,n)}, we know every vertex of $S_n$ is an injection from $\BN^k$ to $\BN^k \times [n]$ that is a translation on rays. Since there is exactly one $k$-dimensional ray in a finite partition of $\BN^k$ into rays, we can track to which copy of $\BN^k$ this ray is sent. If we assume the vertex set of $\Delta^{n-1}$ is $[n]$, this gives us a map on the vertices, which uniquely extends to a simplicial map. 

We want to prove that $\pi^{(n-2)}\colon S_n^{(n-2)}\to (\Delta^{n-1})^{(n-2)}$ restricted to the $(n-2)$-skeleton is a fin-retraction. The following lemma will help us to analyze the complex $S_n$.

\begin{lem}
Let $f_1,\dots, f_{p+1}\colon \BN^k \to \BN^k\times [n]$ be $p+1\le n-1$ vertices of $S_n$, then they form a $p$-simplex if and only if their images in $\BN^k\times [n]$ are disjoint.
\end{lem}

\begin{proof}
By the universal property of the coproduct of sets (disjoint union) we can put the $f_i$ together to one map $f\colon \BN^k \times [p+1] \to \BN^k \times [n]$. Because the images of the maps are disjoint and every individual map is injective, $f$ is also injective. By Lemma  \ref{lem:hom(m,n)} $f\in \Hom_{\CH}(m,n)$, which proves the lemma by Definition \ref{defn:Sn(X,A)}.
\end{proof}

\begin{prop}
The restriction of $\pi^{(n-2)}$ to the $(n-2)$-skeleton maps surjectively to the $(n-2)$-skeleton of $\Delta^{n-1}$. Furthermore $\pi^{(n-2)}\colon S_n^{(n-2)} \to(\Delta^{n-1})^{(n-2)}$ is a fin-retraction. 
\end{prop}

\begin{proof}
$\pi^{(n-2)}$ is clearly surjective, because the identity map in $\Hom_{\CH}(n,n)$ defines an $(n-1)$-simplex that maps surjectively to $\Delta^{n-1}$. Thus the faces of this $(n-1)$-simplex map surjectively onto the $(n-2)$-skeleton of $\Delta^{n-1}$.

Let $S$ be a finite set of vertices of $S_n$. We want to inductively construct vertices $f_1, \dots, f_n \in S_n$ such that $\rho\colon (\Delta^{n-1})^{(n-2)} \to S_n^{(n-2)}$ mapping $i$ to $f_i$ gives a section of $\pi$, i.e.\ every $n-1$ vertices form an $(n-2)$-simplex in $S_n$ and $\pi(f_i) = i$. We will then also prove that $\rho$ is an $S$-section.

Assume we have all $f_i$ with $i< p$ already constructed. To construct $f_p$, we consider $\BN^k$ as a $k$-dimensional ray and send it to $\BN^k \times  \{p\}$ translating it far enough out (choosing $d\in \BZ^k$ large enough) that its image is disjoint from all images of the vertices in $\{f_i\mid i < p\}\cup S - \pi^{-1}(p)$. It immediately follows that the images $f_1(\BN^k), \dots, f_n(\BN^k)$ are pairwise disjoint. Thus $\rho$ gives a section.

Let $\sigma$ be a simplex in the span of $S$ and $\tau$ a simplex in the $(n-2)$-skeleton of $\Delta^{n-1}$. $\tau$ lies in the link of $\pi \sigma$ if and only the set of vertices of $\pi\sigma$ and  $\tau$ are disjoint and the union has at most cardinality $n-1$. If this is the case the vertices of $\sigma$ and $\rho\tau$, which are some $f_i$, form a simplex in $S_n^{(n-2)}$ by the previous lemma. Vice versa, if the vertices of $\sigma$ and $\rho\tau$ form a simplex in $S_n^{(n-2)}$ it can at most have $n-1$ vertices, and the $k$-dimensional ray of each vertex must be sent to a different copy of $\BN^k$. That means that $\pi$ sends the vertices to distinct vertices.
\end{proof}

\begin{cor} \label{cor:H3}
The simplicial complex $S_n$ is $(\frac{n-3}{2})$-connected.
\end{cor}

\begin{proof}
From Proposition \ref{prop:f-join} we follow that $S_n^{(n-2)}$ is $(n-3)$-connected. For $n\ge 2$, $n-3\ge \left\lfloor\frac{n-3}2\right\rfloor$ implies it is in particular $(\frac{n-3}{2})$-connected. Clearly $S_n$ is always non-empty, which was left to show for $S_1$.
\end{proof}

The corollary proves that $S_{n+1}$ is $\frac{n-2}{2}$-connected, and hence \ref{cond:H3}  for  $S_n(1,1)$. Now by Theorem \ref{thm:hom stab}, this finishes the proof of Theorem A.

\section{Representation stability}\label{section:rep stab}

In this last section, we want to analyze the homology of the Houghton 
groups $\mathcal H_{k,n}$ and prove Theorem B that with constant noetherian 
coefficients their homology can be described as ``essentially'' finitely 
generated $\FI$-modules. We will shortly explain how we consider the 
homology of the Houghton groups as $\FI$-modules. Then describe the 
central stability theory of Putman and Sam, which they introduced in 
\cite{PS}. Making slight adjustments, we can use their Theorem 5.13 to 
prove that the homology of the Houghton groups are essentially finitely 
generated $\FI$-modules.

Let $\FI$ be the category of \textbf{f}inite sets and 
\textbf{i}njections. In fact this is the homogeneous category we get 
from the construction in Theorem \ref{thm:hom cat construction} starting 
with the symmetric monoidal groupoid of finite sets and bijections. (See 
also \cite[Section 5.1]{RWW}.) Let us fix a commutative noetherian ring 
$R$, then we call a functor from $\FI$ to the category of $R$-modules an 
$\FI$-module. Note that $\FI$ is equivalent to the full subcategory only 
defined on the objects $[n]$ for $n\ge 0$.\footnote{Here 
$[0]:=\emptyset$.} Likewise the category of $\FI$-modules is equivalent 
to the functor category from the mentioned subcategory to the category 
of $R$-modules. By abuse of notation we will from now on refer to this 
subcategory when we write $\FI$.

Let us define a functor $G$ from $\FI$ to the category of groups by 
assigning $n$ to the Houghton groups $G(n) = \mathcal H_{k,n}$. Given an 
injection $f\colon [m]\hookrightarrow[n]$, we define $f_*\colon \mathcal 
H_{k,m} \to \mathcal H_{k,n}$ be sending $g\in \mathcal H_{k,m}$ to the element in 
$\mathcal H_{k,n}$ that maps $(x,f(i))\mapsto (y,f(j))$ if $g(x,i) = (y,j)$ 
and leaves all other elements fixed. One easily checks that this 
assignment is functorial. Postcomposing with the group homology functor 
$H_i(-;R)$ will thus give an $\FI$-module, where $n$ is sent to 
$H_i(\mathcal H_{k,n};R)$.

Putman and Sam work with complemented categories in \cite{PS} which can be shown to be homogeneous categories. A \emph{complemented category} is a symmetric monoidal category $(\mathcal C, \oplus, 0)$ with the following properties.
\begin{enumerate}
\item Every morphism in $\mathcal C$ is a monomorphism.
\item $0$ is initial. 
\item $\Hom(A\oplus B, C) \to \Hom(A,C) \times \Hom(B,C)$ given by $f \mapsto (f\circ (\id_A \oplus \imath_B), f\circ (\imath_A\oplus \id_B))$ is injective.
\item Every subobject has a unique complement.
\end{enumerate}
A \emph{subobject} of an object $X$ is an equivalence class of monomorphisms to $X$. Monomorphisms $f\colon A\to X$ and $f'\colon A' \to X$ are equivalent if there is an isomorphism $\psi\colon A \to A'$ such that $f = f'\circ \psi$. A \emph{complement} of a suboject of an object $X$ is a subobject $g\colon B \to X$ if there is an isomorphism $\phi \colon A \oplus B \to X$ such that $f = \phi \circ ( \id_A \oplus \imath_B)$ and $g = \phi \circ (\imath_A \oplus \id_B)$.

Putman and Sam say a monoidal category has a \emph{generator} $X$ if all objects are isomorphic to $X^{\oplus n}$ for a unique $n\in\mathbb N$. We can then speak of the \emph{$X$-rank} of an object. Let $\mathcal B$ and $\mathcal C$ be complemented categories with generators $X$ and $Y$, respectively. Putman and Sam call a strong monoidal functor $\Psi\colon \mathcal B \to \mathcal C$ a \emph{highly surjective} functor if $\Psi(X) = Y$ and $\Psi_* \colon \Aut_{\mathcal B}(B) \to \Aut_{\mathcal C}(\Psi(B))$ is surjective for all objects $B \in \mathcal B$. In \cite[Section 5]{PS} it is proven that then there is a $\mathcal C$-module $\mathcal H_i(\Psi;R)$ with
\[ \mathcal H_i(\Psi; R)_{Y^n} = H_i( \ker(\Aut_{\mathcal B}(X^n) \to \Aut_{\mathcal C}(Y^n));R).\]

The homogeneous categories $\FI$ and $\CH$ which we defined in Section \ref{section:hom stab} are complemented categories. For $\FI$ this is stated in \cite[Example 1.10]{PS}. For $\CH$ it is not much harder to see. (a) is clear because all morphisms are injective maps. $\CH$ being a homogeneous category implies (b). A morphism in $\Hom_{\CH}(a\oplus b,n)$ is a map from $\BN^k \times[a]\, \amalg\, \BN^k \times [b] \to \BN^k \times [n] $. (c) follows from $\amalg$ being the coproduct of sets.
Finally for (d) we first observe that two maps $f\colon \BN^k \times [m] \to \BN^k \times [n]$ and $f' \colon \BN^k\times [m'] \to \BN^k \times [n]$ represent the same subobject of $n$ if and only if their image is in $\BN^k \times [n]$ is the same. Clearly the image is an invariant of a subobject. On the other hand if two such maps have the same image, $m=m'$ is the number of $k$-dimensional rays that fit into the image. If $f$ and $f'$ are translations on every ray of finite ray partitions of $\BN^k\times [m]$ then this gives two finite ray partitions of the image. The intersection yields a common refinement which is again a finite ray partition of the image. Thereby we find two finite ray partitions of $\BN^k \times [m]$ that by $f$ and $f'$, respectively, map to this refinement by translations on the rays. Thus $f^{-1} \circ f' \in \tilde{\mathcal H}_{k,m}$ and both maps represent the same subobject. The complement of a subobject $f\colon \BN^k \times [m]\to \BN^k\times[n]$ is then uniquely given by $\BN^k\times[n] - \im f$ which can be partitioned into finitely many rays.

Similar to the discussion in the introduction we can find a functor $\Psi\colon \CH \to \FI$ sending a morphism $f\colon \BN^k \times[m] \to \BN^k \times [n]$ to the injection given by the map where the $m$-many $k$-dimensional rays are sent to. This functor is in fact highly surjective, since the generator $1$ is mapped to the generator $1$ and $\tilde{\mathcal H}_{k,n} \to \mathfrak S_n$ is surjective.

Furthermore Putman and Sam's $\mathcal H_i(\Psi; R)$ for $\Psi\colon \CH \to \FI$ coincides with the $\FI$-module we have defined above. 

Let us introduce a notation to truncate modules $V$ over a complemented category $\mathcal C$ with generator $Y$. By 
$\trunc_{\ge k} V$, we mean the functor that sends all objects with $Y$-rank $n<k$ to zero and 
all other objects $A$ to $V_A$ as before. We call $V$ \emph{essentially finitely 
generated} if there is some $k\in \BN$ such that $\trunc_{\ge k} V$ is 
finitely generated. In \cite[Section 3]{PS} Putman and Sam introduce a complex of $\mathcal C$-modules
\[ \Sigma_* \colon \cdots \to \Sigma_2 V \to \Sigma_1 V \to V\]
where $(\Sigma_p V)_{Y^n}$ is given by $\Ind_{\Aut(Y^{n-p})}^{\Aut(Y^n)} V_{Y^{n-p}}$ and hence only depends on $V_{Y^{n-p}}$. Hence 
\[(\Sigma_p \trunc_{\ge k} V)_C = (\Sigma_p V)_C\]
for all objects $C$ with $Y$-rank $n\ge p+k$.
In the light of this observation, we can generalize their Lemma 3.6 and Theorem 3.7 to the following.

\begin{lem}\label{lem:3.6}
Let $\mathcal C$ be a complemented category with generator $Y$ and let $V$ be a $\mathcal C$-module over a ring $R$. Assume that all $V_C$ with $C$ having large enough $Y$-rank are finitely generated $R$-modules. Then $V$ is essentially finitely generated if and only if $d\colon (\Sigma_1 V)_C \to V_C$ is surjective for all $C$ with large enough $Y$-rank.
\end{lem}

\begin{thm}\label{thm:3.7}
Let $\mathcal C$ be a complemented category with generator $Y$. Assume that the category of $\mathcal C$-modules is noetherian, and let $V$ be an essentially finitely generated $\mathcal C$-module. Fix some $q\ge 1$. For all $C$ with large enough $Y$-rank the chain complex
\[ (\Sigma_q V)_C \to (\Sigma_{q-1} V)_C \to \dots \to (\Sigma_1 V)_C \to V_C \to 0\]
is exact. 
\end{thm}

We need one more piece of notation. Putman and Sam define a semisimplicial set $\mathfrak I_C$ for a complemented category $\mathcal C$ with generator $Y$ by
\[ \mathfrak I_{C,p} = \Hom_{\mathcal C}(Y^{p+1}, C).\]
For $C = Y^n$ this is the same semisimplicial set $W_n(0,Y)$ defined by Randal-Williams and Wahl in \cite[2.1]{RWW}.

Let us formulate a slight variation of Theorem 5.13 from \cite{PS} that we will use to prove  that for every $i\ge 0$ the $\FI$-module given by $H_i(\mathcal H_{k,n}; R)$ is essentially finitely generated for every noetherian ring $R$.

 \begin{thm}\label{thm:5.13}
Let $\mathcal B$ and $\mathcal C$ be complemented categories with generators $X$ and $Y$, respectively. Let $\Psi\colon \mathcal B \to \mathcal C$ be a highly surjective functor. Fix a noetherian ring $R$ and assume the following conditions.
\begin{enumerate}
\item\label{item:FI-mod noetherian} The category of $\mathcal C$-modules is 
noetherian.
\item\label{item:H_i(G_n) fg} For all $i\ge 0$ the $R$-module 
$\mathcal H_i(\Psi;R)_{C}$ is a finitely generated $R$-module for all $C$ with large 
enough $Y$-rank.
\item\label{item:H3} Fix $q\ge 0$. Then $\mathfrak I_B$ is $q$-acyclic for all objects $B\in \mathcal B$ with large enough $X$-rank. 
\end{enumerate}
Then $\mathcal H_i(\Psi;R)$ is an essentially finitely generated $\mathcal C$-module for all $i\ge 0$.
\end{thm}

In Putman and Sam's proof of their Theorem 5.13 one can now replace their Lemma 3.6 and Theorem 3.7 by our Lemma \ref{lem:3.6} and Theorem \ref{thm:3.7}, and the word finitely generated by essentially finitely generated to get a proof of Theorem \ref{thm:5.13}.

Theorem B is an application of Theorem \ref{thm:5.13} if we set $\mathcal B = \CH$, $\mathcal C=\FI$, and $\Psi \colon \CH \to \FI$ as above. 
Let us check the conditions. The first condition was proved by Church, Ellenberg, Farb, and Nagpal in 
\cite[Theorem A]{CEFN}. The second condition can be directly derived 
from $\mathcal H_{k,n}$ being $FP_{n-1}$, i.e.\ the trivial $\BZ\mathcal 
H_{k,n}$-module $\BZ$ admits a projective resolution which is finitely 
generated in dimensions $\le n-1$. This property is proved in 
\cite[5.1]{B} for the original Houghton groups and in \cite[Theorem B]{BH} for $k\ge 2$. The last condition is Corollary \ref{cor:H3} together with Remark \ref{rem:semisimplicial}.

\bibliographystyle{amsplain}

\end{document}